\documentclass{amsart}
\pdfoutput=1
\usepackage{fullpage, amsmath, amsthm, amsfonts, amssymb, amscd}
\usepackage{mathrsfs}
\usepackage{hyperref}
\usepackage{stmaryrd} 

\newcommand{\isom}{\cong}

\newcommand{\ssq}{\subseteq}
\newcommand{\ssnq}{\subsetneq}

\renewcommand{\o}{\varnothing}
\newcommand{\ul}[1]{\underline{#1}}
\newcommand{\comment}[1]{}

\newcommand{\V}{\mathbb{V}}

\renewcommand{\char}{\text{char}}
\renewcommand{\div}{\text{div}}

\def\a{\mathfrak{a}}
\def\b{\mathfrak{b}}
\newcommand{\m}{\mathfrak{m}}
\newcommand{\n}{\mathfrak{n}}
\newcommand{\p}{\mathfrak{p}}

\newcommand{\shF}{\mathscr{F}}

\newcommand{\shJ}{\mathscr{J}}

\newcommand{\fpow}[1]{\ensuremath ^{[p^{#1}]}}
\newcommand{\frob}{\ensuremath ^{[p]}}

\newcommand{\omop}{\ensuremath ^{1-\frac{1}{p}}}

\DeclareMathOperator{\Hom}{Hom}

\newcommand{\tensor}{\otimes}

\newcommand{\mb}{\mathbb}

\newcommand{\N}{\mb{N}}

\newcommand{\bF}{\mb{F}}

\def\al{\alpha}

\def\s{\sigma}

\renewcommand{\phi}{\varphi}

\newtheorem{thm}{Theorem}[section]
\newtheorem*{thm*}{Theorem}
\newtheorem{coro}[thm]{Corollary}
\newtheorem{prop}[thm]{Proposition}
\newtheorem{lemma}[thm]{Lemma}

\theoremstyle{remark}

\theoremstyle{definition}
\newtheorem{remark}[thm]{Remark}
\newtheorem{exa}[thm]{Example}

\newtheorem{defn}[thm]{Definition}

\newtheorem{ackn}[thm]{Acknowledgements}

\newenvironment{pfClaim}{\,{\em Proof of claim:}}{\hfill$\boxtimes$\newline\newline}

\title{A Note on Injectivity of Frobenius on Local Cohomology of Hypersurfaces} 
\author{Eric Canton}

\begin{document}
\begin{abstract}
  Let $k$ be a field of characteristic $p > 0$ such that $[k:k^p] < \infty$ and let $f \in R = k[x_0, \dots, x_n]$ be homogeneous of degree $d$. We obtain a sharp
  bound on the degrees in which the Frobenius action on $H^n_\m(R/fR)$ can be injective when $R/fR$ has an isolated non-F-pure point at $\m$. As a corollary,
  we show that if $(R/fR)_\m$ is not F-pure then $R/fR$ has an isolated non-F-pure point at $\m$ if and only if the Frobenius action is injective in degrees
  $\le -n(d-1)$. 
\end{abstract}
\maketitle
\section{Introduction}

Let $k$ be a field of characteristic $p$ such that $[k:k^p] < \infty$ and let $f \in R = k[x_0, \dots, x_n]$ be a homogeneous polynomial of degree $d$. For 
simplicity, assume that the test ideal $\tau(f\omop) = \m^j$ for some $j \ge 1$, where $\m = (x_0, \dots, x_n)$. 
Our main theorem obtains the following sharp bound on the degrees in which the Frobenius action on $H^n_\m(R/fR)$ is injective.

\begin{thm*}[Theorem {\ref{mainthm}}]
  If $\tau(f\omop) = \m^j \ssq \m$, then the below Frobenius action is injective: 
  \[ F: H^n_\m(R/fR)_{< -n-j + d} \to H^n_\m(R/fR)_{<p(-n-j+ d)}. \]
\end{thm*}

Our assumption that $\tau(f\omop) = \m^j$ implies that while $(R/fR)_\m$ is not F-pure, $(R/fR)_\p$ is F-pure for every prime $\p\ssnq\m$. We say such 
rings have an isolated non-F-pure point at $\m$. The study of F-pure rings has a long history and their theory is rich: 
Hochster and Roberts first defined F-pure rings and explored the relationship of F-purity to local cohomology (and the Frobenius action thereof) in \cite{HochRob76}.
Fedder continued this program of study, obtaining a criterion for F-purity and showing the equivalence of F-purity and F-injectivity for local Gorenstein rings of 
characteristic $p$ \cite{Fed}.

A corollary to our main theorem is that when $(R/fR)_\m$ is not F-pure, $R/fR$ has an isolated non-F-pure point at $\m$ if and only if Frobenius acts 
injectively in sufficiently negative degrees. Moreover, the degree in which it must be injective depends only on the degree of $f$. 

\begin{thm*}[Corollary {\ref{corollary}}]
  If $(R/fR)_\m$ is not F-pure then $R/fR$ has an isolated non-F-pure point at $\m$ if and only if the below Frobenius action is injective:
  \[ F: H^n_\m(R/fR)_{\le -n(d-1)} \to H^n_\m(R/fR)_{\le -pn(d-1)}. \]
\end{thm*}

In their study of the F-pure thresholds of Calabi-Yau hypersurfaces, Bhatt and Singh proved a similar result \cite[Theorem 3.5]{BS} under the assumption
that $R/fR$ has an isolated singularity at $\m$. Their methods generalize well to the setting of this paper. The relationship between isolated singularities and
isolated non-F-pure points is as follows: regular rings are F-pure, so $\{\text{non-F-pure points of $R/fR$}\} \ssq \V(f)_{sing}$. Thus if $(R/fR)_\m$ is not 
F-pure and has an isolated singularity, it follows that it has an isolated non-F-pure point. Interesting examples of these phenomena often arise as
affine cones over smooth projective varieties. 

\begin{ackn}
  I want to thank my advisor Wenliang Zhang for suggesting this problem to me and for useful discussions. 
\end{ackn}

\section{Main result}
The {\em Frobenius map} on a ring $A$ of prime characteristic $p > 0$ is the ring homomorphism $F: A \to A$ given by $F(a) = a^p$. We say that $A$ is 
F-finite if $A$ is a finitely generated module over $F(A) = A^p$. 

We fix notation: throughout, $k$ will denote an F-finite field of characteristic $p > 0$. Let $R = k[x_0, \dots, x_n]$ be the polynomial ring 
in $n+1$ variables over $k$ and $f \in R$ be homogeneous of degree $d$. Note that in this case $R$ is F-finite. Several of the definitions we provide 
(\ref{defnPairs}, \ref{defnNonFP}) rely on F-finiteness for their equivalence to other definitions in the literature.  
Denote by $\m = (x_0, \dots, x_n)$ the homogeneous maximal ideal of $R$. For an ideal $I \ssq R$ and a natural number $e \in \N_0$ we denote by 
$I\fpow{e} = (u^{p^e} \, | \, u \in I)$.  

We recall a special case of the test ideals introduced by Hara and Yoshida \cite{HY} and extended to pairs by Takagi \cite{Tak}. The definition we use is 
\cite[Definition 2.9]{BMS}; proposition 2.22 from the same paper shows the equivalence of the definitions in the regular F-finite case. The test ideal serves
as a positive characteristic analog to the multiplier ideal $\shJ(X, \a^t)$ studied in complex algebraic geometry. 
We refer the reader to \cite[Ch. 9 and 10]{Laz} for an introduction to multiplier ideals. 

\begin{defn}[Test ideal, {\cite[Definition 2.9]{BMS}}]\label{defnPairs}
  The {\em test ideal} $\tau(f^{1-\frac{1}{p^e}})$ is the smallest ideal $\a \ssq R$ such that
  \[ f^{p^e-1}\in \a\fpow{e}. \]
\end{defn}
\begin{remark}\label{testGens}
  Proposition 2.5 from \cite{BMS} gives a useful description of $\tau(f^{1-\frac{1}{p^e}})$: let $\{\lambda_b\}_{b \in B}$ be a basis for $k$ over $k^{p^e}$. 
  The elements $\lambda_b x^{\ul{i}} = \lambda_b x_0^{i_0}\cdots x_n^{i_n}$ with 
  $0 \le i_j \le p^e-1$ and $b \in B$ form an $R^{p^e}$-basis for $R$, so we can express $f^{p^e-1}$ as an $R^{p^e}$-linear combination
  \[ f^{p^e-1} = \sum f_{\ul{i},b}^{p^e} \lambda_bx^{\ul{i}}. \]
  Then the test ideal $\tau(f^{1-\frac{1}{p^e}})$ is the ideal generated by the $f_{\ul{i},b}$ for all $\ul{i}$ and $b$
  appearing above. That is,
  \[ \tau(f^{1-\frac{1}{p^e}}) = (f_{\ul{i},b}\, |\, 0 \le i_j \le p^e-1;\, b\in B). \]
\end{remark}

If the Frobenius map $F: A \to A$ is pure, then we say that $A$ is {\em F-pure.} 
The corresponding notion in characteristic $0$ is that of log canonical (lc) points, and the set of non-lc points is obtained as the vanishing set of the non-lc 
ideal. Fujino, Schwede, and Takagi initiated development of the theory of {\em non-F-pure ideals} in \cite[Section 14]{FST}. As one might expect, the vanishing
locus of the non-F-pure ideal is precisely the set of primes for which $(R/fR)_\p$ fails to be F-pure. We caution the reader that the definition we give is 
specific to the case considered in this note; 
see \cite[Definition 14.4]{FST} for the general definition. 

\begin{defn}[non-F-pure ideal; {\cite[Remark 16.2]{FST}}]\label{defnNonFP}
  The {\em non-F-pure ideal} of $f$, denoted $\s(\div(f))$, is defined to be
  \begin{align*}
    \s(\div(f)) &= \tau(f^{1-\frac{1}{p^e}}) \text{ for $e \gg 0$.}
  \end{align*}
\end{defn}

\begin{prop}\label{radicalProp}
  $\sqrt{\tau(f\omop)} = \sqrt{\s(\mathrm{div}(f))}$. 
\end{prop}
\begin{proof}
  It follows from the definitions that $\s(\div(f)) \ssq \tau(f\omop)$, so it is enough to show that if $\tau(f\omop) \not\ssq \p$ for some prime 
  $\p$ then $\s(\div(f)) \not\ssq \p$. Since $\s(\div(f))$ is the non-F-pure ideal, we check that $(R/fR)_\p$ is F-pure. 

  By assumption, $\tau(f\omop)_\p = R_\p$. Since test ideals localize \cite[Propostion 2.13(1)]{BMS} it follows that $f^{p-1} \not\in (\p R_\p)\frob$. 
  Fedder's Criterion \cite[Theorem 1.12]{Fed} now implies that $(R/fR)_\p$ is F-pure, and so $\s(\div(f)) \not\ssq \p$. 
\end{proof}

\begin{defn}[isolated non-F-pure point]
  We say that $R/fR$ has an {\em isolated non-F-pure point} at $\m$ if $(R/fR)_\m$ is not F-pure but $(R/fR)_\p$ is whenever $\p \ssnq \m$. 
\end{defn}

\begin{remark}
  The vanishing set $\V(\s(\div(f)))$ is precisely the set of points $\p \in \V(f)$ such that $(R/fR)_\p$ is not F-pure. Proposition \ref{radicalProp} now says 
  that the ideal $\tau(f\omop)$ also defines this locus. Therefore, $R/fR$ has an isolated non-F-pure point at $\m$ if and only if $\sqrt{\tau(f\omop)} = \m$.
\end{remark}

\begin{defn}\label{MeDefn}
  Let $e_0\in \N_0$ be the least integer such that $\tau(f\omop) \not\ssq \m\fpow{{e_0}}$. For $e \ge e_0$ define 
  \[ M_e := \min\{\deg(g) \,\,|\,\, g \in (\m\fpow{e}:\tau(f\omop)) \setminus \m\fpow{e} \text{ homogeneous}\}. \]
  Here we adopt the convention $\min \o = \infty$. 
\end{defn}

\newpage

\begin{lemma}\label{MeDecr}
  $M_{e+1} - (n+1)p^{e+1} \le M_{e} - (n+1)p^{e}$ for all $e \ge e_0$. 
\end{lemma}
\begin{proof}
  Note that $M_e = \infty$ for $e\ge e_0$ if and only if $\tau(f\omop) = R$; in this case there is no content to the lemma. Thus we assume that $M_e < \infty$. 
  For simplicity of notation, write $\tau = \tau(f\omop)$. 
  Let $r$ be a homogeneous element of $(\m\fpow{{e}}:\tau)\setminus \m\fpow{{e}}$ with minimum degree 
  $M_{e}$. Then for each term $t$ of every generator $f_{\ul{i},b}$ for $\tau$ (as in Remark \ref{testGens}) we have 
  that $\deg_{x_j}(rt) \ge p^{e}$ for some $0 \le j \le n$. Thus, 
  \begin{align*}
    \deg_{x_j}((x_0\cdots x_n)^{p^{e+1}-p^{e}}rt) &= p^{e+1} - p^{e} + \deg_{x_j}(rt) \ge p^{e+1}
  \end{align*}
  so that $(x_0\cdots x_n)^{p^{e+1}-p^{e}}r \in (\m\fpow{{e+1}}:\tau)$. Since $(\m\fpow{{e+1}}:(x_0\cdots x_n)^{p^{e+1}-p^{e}}) = \m\fpow{{e}}$, 
  we know 
  \[ (x_0\cdots x_n)^{p^{e+1}-p^{e}}r \not\in \m\fpow{{e+1}}. \]
  It follows that $M_{e+1} \le M_{e} + (n+1)(p^{e+1}-p^{e})$. 
\end{proof}

\begin{lemma}\label{MeLemma}
  Assume $(R/fR)_\m$ is not F-pure. Then $R/fR$ has an isolated non-F-pure point at $\m$ if and only if $M_e -(n+1)p^e$ is constant for $e \gg 0$. 
\end{lemma}
\begin{proof}
  For simplicity, write $\tau := \tau(f\omop)$. 
  If $\tau\ssq\m$ then $(\m\fpow{e}:\tau) \ne \m\fpow{e}$ for any $e$, so $M_e < \infty$ for all $e$ in this case. Since we are assuming $(R/fR)_\m$ is not
  F-pure, we conclude that $M_e < \infty$ for all $e$.

  $R/fR$ has an isolated non-F-pure point at $\m$ if and only if 
  $\sqrt{\tau} = \m$, which is equivalent to $\m^\ell \ssq \tau$ for some $\ell \ge 1$.   

  \noindent{\bf Claim:} $(\m\fpow{e}:\tau) \ssq (\m\fpow{e}:\m^\ell)$ for all $e \gg 0$ if and only if $\m^\ell \ssq \tau$. 

  \begin{pfClaim}
    Let $(A, \n)$ be a 0-dimensional Gorenstein local ring and let $L \ssq A$ be an ideal. Write $(-)^\vee$ for the Matlis dual $\Hom_A(-, E_A(A/\n))$
    and note that $A \isom E_A(A/\n)$ since $A$ is 0-dimensional and Gorenstein. Then 
    \begin{align*}
      (0:L) &\isom \Hom_A(A/L, A) \\
            &\isom (A/L)^\vee. 
    \end{align*}
    Now applying the Matlis dual again, we get $A/L \isom (A/L)^{\vee\vee} \isom (0:L)^\vee$ where the first isomorphism follows from finite length of $A/L$. 
    Let $I, J \ssq A$ be two ideals. If $(0:J) \ssq (0:I)$ then we have an exact sequence
    \[ 0 \to (0:J) \to (0:I) \]
    which we dualize to get
    \[ A/I \to A/J \to 0. \]
    Thus, if $A$ is a 0-dimensional Gorenstein ring and $I, J$ are two ideals of $A$ then $(0:J) \ssq (0:I)$ if and only if $I\ssq J$. 

    Note that $R/\m\fpow{e}$ is a 0-dimensional Gorenstein ring for all $e \ge 0$. The above paragraph shows that $(\m\fpow{e}:\tau) \ssq (\m\fpow{e}:\m^\ell)$ 
    if and only if
    $\m^\ell + \m\fpow{e} \ssq \tau + \m\fpow{e}$. For $e \gg 0$, $\m\fpow{e} \ssq \m^\ell$ so this last reads $\m^\ell \ssq \tau + \m\fpow{e}$ for all $e \gg 0$. 
    Therefore
    \[ \m^\ell \ssq \bigcap_{e \gg 0} (\tau + \m\fpow{e}). \]
    This intersection is $\tau$ by Krull's intersection theorem. We conclude that $(\m\fpow{e}:\tau) \ssq (\m\fpow{e}:\m^\ell)$ for $e \gg 0$ if and only if 
    $\m^\ell \ssq \tau$. 
  \end{pfClaim}
  The proof of \cite[Lemma 3.2]{BS} shows that 
  \[ (\m\fpow{e}:\m^\ell) = \m\fpow{e} + \m^{(n+1)p^e - n - \ell} \text{ for $e \gg 0$}. \]
  Thus we have that $\sqrt{\tau} = \m$ if and only if $M_e \ge (n+1)p^e - n - \ell$ for $e \gg 0$ and some $\ell \ge 1$. Lemma \ref{MeDecr} shows that 
  \[ M_{e+1} - (n+1)p^{e+1} \le M_e - (n+1)p^e \]
  for all $e \ge e_0$, so we conclude that $R/fR$ has an isolated non-F-pure point at $\m$ if and only if 
  \[ -n - \ell \le M_e - (n+1)p^e \]
  for some $\ell \ge 1$ and all $e$. Since $\{M_e - (n+1)p^e\}_{e \ge e_0}$ is a nonincreasing sequence of integers, this sequence is bounded below 
  if and only if $M_e - (n+1)p^e$ is constant for $e \gg 0$. 
\end{proof}
\begin{remark}
  If $\tau(f\omop) = \m^j$ for some $j \ge 1$ then the proof shows that in fact $M_e-(n+1)p^e = -n-j$ for all $e\ge e_0$.
\end{remark}

\begin{remark}
  We note that if $M_e < \infty$ then $M_e - (n+1)p^e + d \le 1 + \frac{d}{p} - \frac{n+1}{p}$. Indeed, if $r \not\in \m\fpow{e}$ and
  $\deg(r) = M_e - 1$ then $r\not\in (\m\fpow{e}:\tau(f\omop))$. It follows that $r^pf^{p-1} \not\in \m\fpow{{e+1}}$. This implies
  \[ p(M_e - 1) + (p-1)d \le (n+1)(p^{e+1} - 1). \]
  Dividing both sides by $p$, we have that
  \[ M_e - (n+1)p^e + d \le 1 + \frac{d - (n+1)}{p}. \]
  In particular, as long as $d \le n+1$ or $p > d - (n+1)$ we have that $M_e - (n+1)p^e + d \le 1$. 
\end{remark}

\begin{defn}
  If $R/fR$ has an isolated non-F-pure point at $\m$, define $\delta(f) = M_e - (n+1)p^e$ for $e \gg 0$. 
\end{defn}

Of major importance to our proof of the main theorem is analysis of the following diagram of short exact sequences in local cohomology. 
This appears as \cite[Remark 2.2]{BS}.
\begin{remark}
  For $f \in R$ as above, the Frobenius map $F: R/fR \to R/fR$ fits into a diagram of short exact sequences
  \[\begin{CD}
      0 @>>> R[-d] @>f>> R @>>> R/fR @>>> 0\\
      @.    @V f^{p-1}FVV @V F VV @V F VV @.\\
      0 @>>> R[-d] @>f>> R @>>> R/fR @>>> 0.
  \end{CD}\]
  The long exact sequence in local cohomology now gives
  \[\begin{CD}
      0 @>>> H^n_\m(R/fR) @>>> H^{n+1}_\m(R)[-d] @>f>> H^{n+1}_\m(R) @>>> 0\\
    @.        @VFVV             @Vf^{p-1}FVV          @VFVV             @.\\
      0 @>>> H^n_\m(R/fR) @>>> H^{n+1}_\m(R)[-d] @>f>> H^{n+1}_\m(R) @>>> 0.
  \end{CD}\]
  The rightmost map is injective because $R$ is regular (and so is F-pure), so the snake lemma now implies that injectivity of the map on the left is equivalent to injectivity of 
  the middle map. 
\end{remark}

\begin{thm}\label{mainthm}
  Let $f \in R$ be homogeneous of degree $d$ and assume that $R/fR$ has an isolated non-F-pure point at $\m$. 
  Then the below Frobenius action is injective:
  \[ F:H^n_\m(R/fR)_{<\delta(f) + d} \to H^n_\m(R/fR)_{< p(\delta(f) + d)}. \]
\end{thm}
\begin{proof}
  Writing $N = \delta(f) + d$ we have the diagram in local cohomology
  \[\begin{CD}
      0 @>>> H^n_\m(R/fR)_{< N} @>>> H^{n+1}_\m(R)[-d]_{< N} @>\cdot f>> \cdots\\
        @.    @V F VV                     @V f^{p-1}F VV          @.\\
      0 @>>> H^n_\m(R/fR)_{< pN} @>>> H^{n+1}_\m(R)[-d]_{< pN+d(p-1)} @>\cdot f>> \cdots.
  \end{CD}\]
  As remarked above, injectivity of $F$ on the left is equivalent to that of the 
  middle map $f^{p-1}F$. Assume that we have a homogeneous $0 \ne \al \in H^{n+1}_\m(R)[-d]_{< N}= H^{n+1}_\m(R)_{<\delta(f)}$ such that 
  $f^{p-1}F(\al) = 0$. 
  We have a representation of $\al$ of the form
  \[ \al = \left[\frac{g}{(x_0\cdots x_n)^{p^e}}\right] \]
  with $g \not\in\m\fpow{{e}}$ and where we may assume that the power in the bottom is $p^e$ for some $e \gg 0$ by multiplying by an appropriate form of $1$. 
  Using this representation, we have
  \begin{align*}
    f^{p-1}F(\al) = 0 &\iff f^{p-1}g^p \in \m\fpow{e+1}\\
                      &\iff f^{p-1} \in \left(\m\fpow{e+1}:g^p\right) = \left(\m\fpow{e}:g\right)\frob\\
                      &\iff \tau(f\omop) \ssq (\m\fpow{e}:g) \\
                      &\iff g \in \left(\m\fpow{e}:\tau(f\omop)\right).
  \end{align*}
  Here the equality of colon ideals in the second line follows from Kunz's theorem \cite[Theorem 2.1]{kunz} which says Frobenius is flat if and only if $R$ is 
  regular, along with the fact that if $A \to B$ is a flat ring extension then $(I:_A J)B = (IB:_B JB)$ for any ideals $I, J \ssq A$. Thus, $\deg(g) \ge M_e$ and so
  \[ \deg(\al) = \deg(g) - (n+1)p^e \ge M_e - (n+1)p^e = \delta(f) \]
  This contradicts $\deg(\al) < \delta(f)$. 
\end{proof}
\begin{remark}
  The proof also shows that this bound is optimal: for $e \gg 0$ and an element $r \in (\m\fpow{e}:\tau(f\omop)) \setminus \m\fpow{e}$ homogenous of
  degree $M_e$, if we take $\al = [r/(x_0\cdots x_n)^{p^{e}}]$ then $\al \ne 0$ but $f^{p-1}F(\al) = 0$. 
\end{remark}

\begin{coro}\label{corollary}
  Let $f \in R$ be homogeneous of degree $d$ and assume that $(R/fR)_\m$ is not F-pure. Then $R/fR$ has an isolated non-F-pure point at $\m$ if and only if the 
  below Frobenius action is injective:
  \[ F:H^n_\m(R/fR)_{\le -n(d-1)} \to H^n_\m(R/fR)_{\le -pn(d-1)}. \]
\end{coro}
\begin{proof}
  Assume that $R/fR$ has an isolated non-F-pure point at $\m$. We show that $-n(d-1) < \delta(f) + d$. As in Remark \ref{testGens}, let 
  $\shF = \{f_{\ul{i}, b}\,|\, 0 \le i_j \le p;\, b \in B\}$ be a generating
  set for $\tau(f\omop)$. Since $R/fR$ has an isolated non-F-pure point at $\m$, there exist $n+1$ generators $f_0, \dots, f_n \in \shF$ 
  which form a maximal regular sequence. Write $d_i = \deg(f_i)$. 
  The proof method of \cite[Lemma 3.1]{BS} shows that $\m^{(\sum d_i)- n} \ssq (f_0, \dots, f_n)$.
  Indeed, let $\b = (f_0, \dots, f_n)$. Then the Hilbert series of $R/\b$ is \[ P(R/\b, t) = \prod_{i=0}^n \frac{1-t^{d_i}}{1-t}. \]
  This follows from \cite[Exercise 21.12(b)]{Eis} together with the facts that $P(k[x], t) = \frac{1}{1-t}$ and that $P(M \tensor N, t) = P(M, t)\cdot P(N, t)$ 
  whenever all quantities are defined. The degree of this polynomial is $(\sum_{i=0}^n d_i) - (n+1)$. It follows that there can be no monomials
  of degree greater than $(\sum d_i) - (n+1)$ in $R/\b$. This is equivalent to $\m^{(\sum d_i) - (n+1) + 1} \ssq \b$.

  From this we see that $\left(\m\fpow{e}:\tau(f\omop)\right) \ssq \left(\m\fpow{e}:\m^{(\sum d_i) - n}\right)$ and \cite[Lemma 3.2]{BS} tells us that
  \[ \left(\m\fpow{e}:\m^{(\sum d_i) - n}\right) = \m\fpow{e} + \m^{(n+1)p^e - (\sum d_i)}. \]
  Letting $e \gg 0$ and $r \in \left(\m\fpow{e}:\tau(f\omop)\right) \setminus \m\fpow{e}$ be homogeneous of degree $M_e$, the equality above shows us 
  that
  \[ \deg(r) = M_e \ge (n+1)p^e - \left(\sum d_i\right). \]
  By Lemma \ref{MeLemma} we now conclude $\delta(f) \ge -(\sum d_i)$. Thus, $\delta(f) + d > -(\sum d_i) + d - 1$. Since 
  $d_i = \deg(f_i)$ we have that $pd_i \le d(p-1)$ from which it follows that $d_i < d - 1$. Replacing each $d_i$ with $d-1$ we conclude
  \[ -n(d-1) < -\left(\sum d_i\right) + d - 1 < \delta(f) + d. \]

  Using the contrapositive, if $R/fR$ does not have an isolated non-F-pure point at $\m$, then lemmas \ref{MeDecr} and \ref{MeLemma} tell us 
  $\{M_e - (n+1)p^e\}_{e \ge e_0}$ is unbounded below. 
  If $r \in (\m\fpow{e}:\tau(f\omop)) \setminus \m\fpow{e}$ has degree $M_e$ then $f^{p-1}F([r/(x_0\cdots x_n)^{p^e}]) = 0$ but $[r/(x_0\cdots x_n)^{p^e}] \ne 0$. 
  Letting $e \gg 0$ such that $M_e - (n+1)p^e < -n(d-1)$, we see that the Frobenius action on $H^n_\m(R/fR)_{M_e - (n+1)p^e}$ is not injective. 
\end{proof}

\begin{exa}
  Let $f = x^2y^2 + y^2z^2 + z^2x^2 \in k[x,y,z]$ with $\char(k) > 2$. Then $\tau(f^{1 - \frac{1}{p}}) = \m$ but $f$ does not have an isolated singularity. 
  In this case, the Bhatt-Singh result \cite[Theorem 3.5]{BS} does not apply. Theorem \ref{mainthm} now tells us that the Frobenius action on $H^2_\m(R/fR)$ is 
  injective in degrees $\le 0$. Note that in this case, $H^2_\m(R/fR)_1 \ne 0$ but $H^2_\m(R/fR)_{\ge 2} = 0$ so the Frobenius action on $H^2_\m(R/fR)_1$ is zero.
\end{exa}

\begin{exa}
  We provide an example to show that $M_e - (n+1)p^e$ does not always stabilize at the first step. Let 
  \[ f = x_0^2x_1x_2x_3x_4 + x_0x_1^2x_2x_3x_4 + \cdots + x_0x_1x_2x_3x_4^2 + x_5^6 \in \bF_2[x_0, \dots, x_5]. \]
  Then $\tau(f^{1/2}) = (x_0, x_1, x_2, x_3, x_4, x_5^3)$. Now $M_1 = 5$, $M_2 = 16$, and we see that $M_1 - 6(2) = -7$ but $M_2 - 6(2^2) = -8$. 
  Since $\m^3 \ssnq \tau(f^{1/2})$ we have $-5-3 \le M_e - 6(2^e)$ so we see that $\delta(f) = -8$.
\end{exa}

\bibliographystyle{alpha}
\bibliography{refs}
\end{document}